\documentclass[a4paper]{amsart}
\usepackage[utf8]{inputenc}

\usepackage{tikz}

\usepackage{amsmath, amssymb, amsthm}
\usepackage{mathtools}  
\usepackage{bm}         
\usepackage{dsfont}	
\usepackage{ytableau}   
\usepackage{microtype}  
\usepackage{placeins}

\usepackage{booktabs}   

\usepackage[breaklinks]{hyperref}
\hypersetup{
    colorlinks,
    linkcolor={red!40!black},
    citecolor={blue!50!black},
    urlcolor={blue!80!black}
}

\newtheorem{theorem}{Theorem}
\newtheorem{proposition}[theorem]{Proposition}
\newtheorem{lemma}[theorem]{Lemma}
\newtheorem{corollary}[theorem]{Corollary}
\newtheorem{example}[theorem]{Example}
\theoremstyle{remark}
\newtheorem{remark}[theorem]{Remark}
\theoremstyle{definition}
\newtheorem{definition}[theorem]{Definition}

\DeclareMathOperator{\tr}{tr}

\DeclareMathOperator{\Wg}{Wg}
\DeclareMathOperator{\End}{End}
\DeclareMathOperator{\vecc}{vec}

\newcommand{\ot}[0]{\otimes}
\newcommand{\nn}[0]{\nonumber}
\newcommand{\one}[0]{\mathds{1}}

\newcommand{\N}{\mathds{N}}

\newcommand{\C}{\mathds{C}}

\newcommand{\Z}{\mathds{Z}}
\newcommand{\Q}{\mathds{Q}}

\begin{document}

\title  {Tensor Polynomial Identities}
\date   {\today}

\author {Felix Huber}
\address{ICFO - The Institute of Photonic Sciences,  Av. Carl Friedrich Gauss 3, 08860 Castelldefels (Barcelona), Spain}
\thanks{FH acknowledges support by
the Government of Spain (FIS2020-TRANQI and Severo Ochoa CEX2019-000910-S), 
Fundació Cellex, 
Fundació Mir-Puig, 
Generalitat de Catalunya (AGAUR SGR 1381 and CERCA),
and
the European Union under Horizon2020 (PROBIST 754510).}

\author{Claudio Procesi}
\address{Universit\`a degli Studi di Roma, La Sapienza}

\begin{abstract}
 Tensor polynomial identities generalize the concept of polynomial identities on $d \times d$ matrices to identities on tensor product spaces. 
 Here we completely characterize a certain class of tensor polynomial identities in terms of their associated Young tableaux. 
 Furthermore, we provide a method to evaluate arbitrary alternating tensor polynomials in $d^2$ variables.
\end{abstract}

\maketitle

The aim of this paper is to generalize an identity found by one of the authors~\cite{huber2020positive}: 
for all complex $2\times 2$ matrices $X_1, \dots, X_4$ the following expression vanishes,
$$
 \sum_{\sigma \in S_4} \epsilon_\sigma X_{\sigma(1)} X_{\sigma(2)} \otimes X_{\sigma(3)} X_{\sigma(4)} = 0\,.
$$
Of course this expression is similar to the standard identity where no tensor product $\otimes$ is present, namely
$\sum_{\sigma \in S_4} \epsilon_\sigma X_{\sigma(1)} X_{\sigma(2)} X_{\sigma(3)} X_{\sigma(4)} = 0$ holds
for all complex $2\times 2$ matrices $X_1, \dots, X_4$.

The theory of polynomial identities for algebras is a long established subject, see~\cite{agpr}.
Here we extend the theory to alternating tensor polynomials.
An article considering the more general tensor trace identities is in preparation~\cite{pro2020}.
\smallskip

For simplicity we start with an associative algebra $A$ over a field $F$ of characteristic~$0$. 
Denote by $F\langle X\rangle=F\langle x_1,\ldots,x_i,\ldots\rangle$ the free algebra in the variables $x_i$.  By definition the homomorphisms of $F\langle X\rangle$ to $A$  correspond to maps $X\to A$ and can be thought of as {\em evaluations} of the variables $X$ in $A$.  Then 
\begin{definition}\label{PI}
An element $f\in F\langle X\rangle$ is called a {\em polynomial identity} for $A$  (short a PI)  if it vanishes under all evaluations of $X$ in $A$.
\end{definition} 
The elements of $F\langle X\rangle$ are usually called {\em non commutative polynomials}.
Given any positive integer $k$ we may consider the algebra $A^{\otimes k}$ and then for every evaluation  $\pi:F\langle X\rangle\to A$ of $X$ in $A$ we have a corresponding evaluation  $\pi^{\otimes k}:F\langle X\rangle^{\otimes k}\to A^{\otimes k}$.
\begin{definition}\label{tPI}
The elements of $F\langle X\rangle^{\otimes k}$ will be called {\em tensor polynomials}.  
They can be thought of as polynomials in the {\em tensor  variables}  $x_j^{(i)}:=1^{\otimes i-1}\otimes x_j  \otimes  1^{\otimes  k-1 }$.
An element $f\in F\langle X\rangle^{\otimes k}$ is called a {\em tensor  polynomial identity} for $A$  (short a TPI)  if it vanishes under all evaluations of $X$ in $A$.
\end{definition} 
Given a positive integer $d$  and a field $F$ denote by $M_d(F)$ the algebra of $d\times d$ matrices with entries in $F$.  We will in fact often work with $F=\mathbb Q$, since all the formulas we are interested in have integer or rational coefficients,    and then denote  $M_d:=M_d(\mathbb Q).$  
We will start a study of  tensor polynomial  identities $G \in F \langle X\rangle^{\otimes n}$ 
in $k$~variables for  $d\times d$ matrices $M_d$.
Contrary to the theory of polynomial identities little is known about TPI's.  
We want to start this investigation by exhibiting an analogue of the Amitsur--Levitzki Theorem for matrices.  \smallskip

\section{Alternating functions}
We denote by $S_k$  the symmetric group on $k$ elements (usually $1,2,\ldots,k)$.   
Given $\sigma \in S_k$, a permutation, denote by $ \epsilon_\sigma =\pm 1$ its sign.

\begin{definition}\label{anti}
A polynomial function  $G(x_1,\ldots,x_k) $ with values in a vector space $V$ and with $x_i$ some vector variables in some space $U$ is 
\\

1. {\em antisymmetric}, or {\em alternating}, in the variables $x_i$, if 
$$G(x_{\pi(1)},\ldots,x_{\pi(k)})=\epsilon_\pi G(x_1,\ldots,x_k),\ \forall\pi\in S_k.$$

2. {\em linear } in a variable $x_i$ if
$$G(x_1,\ldots,x_i+y_i,\ldots ,x_k)=G(x_1,\ldots,x_i ,\ldots ,x_k)+G(x_1,\ldots, y_i,\ldots ,x_k).$$
\quad $G$ is {\em multilinear} if it is linear in every one of its variables.\\
\\ If $U$ and $V$ are representations of a group $H$ then \\

3. $G$ is  {\em $H$--equivariant} if
$$G(hx_1h^{-1},\ldots,hx_{k}h^{-1})= hG(x_1,\ldots,x_{k})h^{-1},\ \forall h\in H\,.$$
\end{definition}
We will apply this definition to $U=M_d,\ V=M_d^{\otimes n}$ and $H=GL(d)$.

\begin{remark}\label{alv} \phantom{.}
\begin{enumerate}\item If $G(x_1,\ldots,x_k) $ is antisymmetric in two variables, 
then it vanishes when any two variables are evaluated in the same element.
\item If $G(x_1,\ldots,x_k) $ is multilinear and $u_1,\ldots,u_m$ is a basis of $U$ then $G(x_1,\ldots,x_k)$ vanishes on  $U$ if and only if it vanishes when the $x_i$ are evaluated in the elements $u$ of the basis.
\item If $G(x_1,\ldots,x_k)$  is multilinear and    antisymmetric in some subset of the variables with $>m$ elements then it vanishes  on~$U$ if $U$ is of dimension $\leq m$.
\end{enumerate}
\end{remark}

A standard procedure to produce  an alternating polynomial from a given multilinear one $G(x_1,\ldots,x_k)$ in some variables $ X= (x_1,\ldots,x_k) $ is the process $Alt$ of {\em alternation} (usually one adds a denominator $\frac1{k!}$ to make the operator idempotent):
\begin{equation}\label{alter}
Alt_X G(X):= \sum_{\sigma \in S_k} \epsilon_\sigma G(x_{\sigma(1)}, \cdots ,x_{\sigma(k)}).
\end{equation}
The {\em standard polynomial} in $k$ variables is defined as
\begin{equation}\label{st}
 St_k(X) := \sum_{\sigma \in S_k} \epsilon_\sigma x_{\sigma(1)} \cdots x_{\sigma(k)}=Alt_Xx_1x_2\ldots x_k.
\end{equation}

It is clearly, up to a scalar multiple, the unique multilinear and antisymmetric non commutative polynomial in $k$ variables.  This depends upon the fact that the symmetric group $S_k$ acts in a simply transitive way  on multilinear monomials of degree $k$.

\section{Alternating tensor polynomials}
As mentioned, for non commutative polynomials there is, in each degree,  a unique multilinear alternating polynomial, namely the standard polynomial.
However, for multilinear tensor polynomials an analogous statement is no longer true,
since a tensor monomial is of the form (here $m_i$ are ordinary monomials):
\begin{equation}\label{tem}
M=m_1\otimes m_2\otimes \ldots\otimes m_n,\ \deg(m_i)=a_i,\ \deg(M)=\sum_{i=1}^k  a_i.
\end{equation}
The symmetric group $S_k$ still acts in a simply transitive way  on multilinear monomials of degree $k=\sum_ia_i  $  from Formula~\eqref{tem}.
 
So in particular a multilinear tensor polynomial, in $k$ variables,  is a linear combination of terms
corresponding to decompositions of $k$ into $n$~parts. 

For all sequences $\underline a:=(a_1, a_2, \ldots , a_n ), a_i\in\N,\ \sum_ia_i=k$ and $n$, 
we define the element $ST(\underline a)(x_1,\ldots,x_k)\in\Q\langle X\rangle^{\otimes n}$ by alternating a tensor monomial in which the variables appear in increasing order
\begin{align}\label{stlam}
  ST(\underline a)(x_1,\ldots,x_k) &:= Alt_X M_{\underline a}(x_1,\ldots,x_k) \\
  M_{\underline a}(x_1,\ldots,x_k) &:=X_1\otimes X_2\otimes\ldots\otimes X_n\,, \nn \\
X_1&:=x_{1}x_{2}\cdots x_{a_1}\,,\nn\\
X_2&:= x_{a_1+1}x_{a_1+2}\cdots x_{a_1+a_2}\,, \nn\\
\vdots \,\,\,\,&  \nn\\
X_{n-1}&:=  x_{k- a_{n-1} -a_n + 1} x_{k- a_{n-1} -a_n + 2} \cdots x_{k-a_n}\,. \nn\\
X_n&:=  x_{k-a_n + 1}x_{ k-a_n+2}\cdots x_{ k}\,. \nn
\end{align}
 
In order to study  which of the tensor polynomials $ST(\underline a)$ are  identities we may assume that, 
up to a permutation of the tensor factors,
the sequence $a_i$ is non--increasing and hence a partition $\lambda\vdash k$.
In fact define $\tilde{\underline a} := (i,j) \underline a$ obtained by exchanging the elements $i$ and $j$ and of $\underline a$.
We then have:
 
\begin{proposition}\label{cse}
 $ {ST}(\tilde{\underline a})  =\epsilon\cdot  (i,j) ST(\underline a) (i,j),\ \epsilon=(- 1)^{a_ia_j} $, obtained by exchanging the tensor factors $i$ and $j$. Thus $ {ST}(\tilde{\underline a})$ is a TPI if and only if  $ ST(\underline a) (i,j) $  is a TPI.\end{proposition}
\begin{proof}  
 The sign $\epsilon$ depends upon the fact that, for a suitable $\sigma\in S_k$  we have   
 $$ M_{\underline {\tilde a}}=  M_{\underline a}(x_{\sigma (1)},\ldots,x_{\sigma(k)}),\quad   \text{e. g.}\   (1,2) x_1\otimes x_2(1,2)=x_2\otimes x_1=M_{(1,1)}(x_2,x_1). $$
 We need to determine the sign of $\sigma$.
 
The sequence of variables affecting positions $i$ and $j$ 
changes in $ST(\underline a)$ to $ST(\tilde{\underline a})$ is like that of cards in a deck when it is cut:
Now if both parts $ a_i$ and $a_j$ are of even size, then the original variable order can be restored by applying 
an odd permutation (specifically, an odd cycle) an even number of times, yielding $(+1)$.
If both parts $a_i$ and $a_j$ are of odd size, the original order of variables can be restored by applying 
an odd permutation and odd number of times, inducing a $(-1)$.
If one part has even length and the other odd, the original order of variables is restored by applying even permutations, still giving a $+1$ sign.
Thus a sign of $(-1)$ is only obtained when both $a_i$ and $a_j$ are of odd size, and $(+1)$ otherwise. This ends the proof.
\end{proof}

  One may also allow some  $ a_i=0$, where the empty product of variables is 1. 
For instance
$$ST(2,0,1 ) = \sum_{\sigma \in S_{ 3}}\epsilon_\sigma  x_{\sigma(1)}x_{\sigma(2)}\otimes 1\otimes         x_{\sigma(3)}. $$  

We define the partion $m^n = \underbrace{(m, \dots, m)}_{n \text{ times}}$
so we have the balanced element
 
\begin{equation}\label{ba}
 ST(m^n)(x_1, \dots, x_{mn}) :=  Alt_X X_1\otimes X_2\otimes\ldots\otimes X_n\,,
\end{equation}
where each $X_i$ has the same size, 
$X_1:=x_{1}x_{2}\cdots x_{m}$, $X_2:= x_{m+1}x_{m+2}\cdots x_{2m}$,\dots, $X_n:=  x_{(n-1)m+1}x_{(n-1)m+2}\cdots x_{mn}$.

\smallskip
The Amitsur--Levitzki Theorem states that for $d\times d$ matrices, $ St_{2d}(X)$ is a PI and of minimal degree. 
For tensor polynomials we ask  for an $n$--fold tensor power  of $M_d(F)$ and a partition 
$\lambda \vdash k$ wether $ST(\lambda)(x_1,\ldots,x_k)$ is a tensor polynomial identity. 
In particular, which is the minimum $m$  for which $ST(m^n)$ is a TPI? (This will be answered in Proposition \ref{ad}.)

Remark~\ref{alv} makes clear that if $k>d^2$, then all $ST(\lambda),\ \lambda\vdash k$ and $ST(m^n), mn=k$ are TPI
for the algebra of $d\times d$ matrices $M_d(F)$. It thus remains to analyze the case of $k\leq d^2$.

\bigskip

It is well known that:
\begin{proposition}\label{ilmulti}
A  multilinear and antisymmetric $  \C$ valued polynomial  function $g (x_1,\ldots,x_m)$ in $m$ variables $ x_i\in \mathbb C^m$ is a multiple    
of the determinant of the matrix with the $x_i$ as rows,
\begin{equation}
 g(x_1,\ldots,x_{m})  = u \det(x_1, \dots, x_m)\,, \quad u \in  \C .\nonumber
\end{equation} 
\end{proposition}  
Note that this is equivalent to the fact that for an $m$-dimensional vector space~$V$  we have $\dim \bigwedge^m V=1 $.

For a multilinear and antisymmetric vector valued polynomial $G(x_1,\ldots,x_{m}) \in W$,\ $x_i\in\C^m$,  which maps to a vector space $W$, 
each coordinate has the same property. Thus $$G(x_1,\ldots,x_{m})=w\det(x_1,\ldots,x_{m}),\quad w\in W$$
where $w$ is some constant vector in $W$.

We apply this to $M_d(F)$, a $d^2$ dimensional vector space,  and identify $M_d(F)\cong F^{d^2}$ using the canonical basis  of  elementary matrices $e_{i,j}$  ordered lexicographically.
For example, 
$$d=2,\quad e_{1,1},\  e_{1,2},\  e_{2,1},\  e_{2,2}. $$
Thus given $d^2$ matrices  $ x_1,\ldots, x_{d^2}\in M_d(F)$ 
we consider each matrix $x_i$ as a column vector with $d^2$ coordinates $\vecc(x_i) \in F^{d^2}$.
The {\em determinant} is defined by
\begin{equation}\nonumber
 \det(x_1, \dots, x_{d^2})= 
 \det\begin{pmatrix}
 \vecc(x_1) & \cdots & \vecc(x_{d^2})
 \end{pmatrix}.
\end{equation} 

\begin{remark}\label{coni}
The linear operators on matrices $L_g:x\mapsto gx$ and $R_g:x\mapsto xg $ have determinant   $\det(L_g)=\det(R_g)=\det(g)^m$  so $\det(L_gR_g^{-1})=1$. Therefore $\det(x_1, \dots, x_{d^2})$ is $GL(d)$-invariant,
\begin{equation}\label{eq:det_GLinv}
 \det(g x_1 g^{-1}, \dots, g x_{m} g^{-1})
= \det(x_1, \dots, x_{m} ),\quad \forall g \in GL(d)\,.
\end{equation}
\end{remark}

Denote by $\Sigma_n(F^d)$ the algebra spanned by the permutations in $M_d(F)^{\otimes n}=\End((F^d)^{\otimes n})$.

From Proposition~\ref{ilmulti}, for a multilinear alternating tensor polynomial $G( x_1,\ldots,x_{d^2})$ $\in F\langle  X\rangle^{\otimes n}$ 
in $d^2$ matrix variables evaluated in the algebra $M_d(F)$ of $d\times d$ matrices we have:
\begin{proposition}\label{lelej}
There is an element $J_G\in M_d(F)^{\otimes n}$ invariant under the diagonal action of the linear group $G=GL(d)$ such that 
\begin{equation}\label{JG}
G(x_1,\ldots,x_{d^2})=\det(x_1,\ldots,x_{d^2})J_G,\quad J_G\in \Sigma_n(F^d)\subset M_d(F)^{\otimes n} . 
\end{equation}
\end{proposition}
\begin{proof} Formula \eqref{JG} is a special case of  Proposition \ref{ilmulti}, we only need to prove that  $J_G\in M_d(F)^{\otimes n}$ is invariant.       

 Naturally, non commutative polynomials are equivariant. 
Consequently, also any non commutative tensor polynomial  $G(x_1,\ldots,x_{k})\in F\langle  X\rangle^{\otimes n}$ 
is {equivariant} and thus it is a map  $M_d(F)^k\to M_d(F)^{\otimes n}$ satisfying
\begin{equation}
 G(gx_1g^{-1},\ldots,gx_{k}g^{-1})= gG(x_1,\ldots,x_{k})g^{-1},\quad \forall g\in GL(d) \,. \nonumber
\end{equation} 
 Here $g$ acts on $(F^d)^{\otimes n}$ diagonally, i.e. $g (y_1 \ot \dots \ot y_n) = g y_1 \ot \dots \ot g y_n$.
Since  from Formula \eqref{eq:det_GLinv} we have that $\det(x_1,\ldots,x_{d^2})$ is  $GL(d)$ invariant, 
it follows that $J_G$ is a constant equivariant map and hence an invariant.
  
  From the Schur-Weyl duality it is known that the elements $a\in M_d(F)^{\otimes n}=\End((F^d)^{\otimes n})$ invariant under the action of the linear group are the linear combinations of the elements of the symmetric group $S_n\subset   M_d(F)^{\otimes n}$.  
 \end{proof} 
  That is, any evaluation of $G$ is a linear combination of permutations,  with coefficients multiples of the determinant, Formula~\eqref{JG}.
 
\begin{equation}\label{jjg}
J_G = \sum_{\sigma \in S_n} c_\sigma \sigma \quad \text{with} \quad c_\sigma \in \C \,.
\end{equation}
\begin{remark}\ \\
(1)\quad In particular if $\lambda\vdash d^2$  and $G=ST(\lambda)$  we will denote  $J_\lambda:=J_G$ in Eq.~\eqref{jjg}.
We may apply this also to sequences because with Proposition~\ref{cse} we have 
\begin{equation}\label{conJ}
    J_{\tilde{\underline a}}  = (- 1)^{a_ia_j} \cdot  (i,j) J_{ \underline a } (i,j). 
\end{equation} 
 (2)\quad Since the algebra $\Sigma_k(F^d)$ spanned by the permutations in $\End((F^d)^{\otimes n})$ is semisimple, 
an element $\alpha\in \Sigma_k(F^d)$ is equal to 0 if and only if  $\tr(\sigma \alpha)=0,\ \forall \sigma\in S_n$. 
In fact we shall see in Section~\ref{sect:inWein} how one can recover $\alpha$  from the element $\Phi(\alpha):=\sum_{\sigma\in S_n}\tr(\sigma^{-1} \alpha)\sigma \in \Sigma_k(F^d)$.
\end{remark}

\section{The invariant exterior algebra $(\bigwedge M_d(F)^*)^G$}
Recall that the exterior algebra $\bigwedge U^*$, with $U$ a vector space,   
can be thought of as the space  of multilinear  alternating functions on $U$. 
Then exterior multiplication as functions is given by the Formula:
 $$f(x_1,\ldots,x_h)\in\bigwedge^h U^* ;g(x_1,\ldots,x_k)\in \bigwedge^k U^*,$$
$$
\ f\wedge g(x_1,\ldots , x_{h+k}):=\frac{1}{h!k!}\sum_{\sigma\in S_{h+k}}f(x_{\sigma(1)},\ldots,x_{\sigma(h)})  g(x_{\sigma(h+1)},\ldots,x_{\sigma(h+k)})
$$
\begin{equation}\label{ep}
\stackrel{\eqref{alter}}=\frac{1}{h!k!}Alt_{x_1,\ldots,x_{h+k}} f(x_1,\ldots,x_h) g(x_{ h+1 },\ldots,x_{ h+k })\in \bigwedge^{h+k} U^*.
\end{equation}

We recall the following facts. Also see Section $1.24$ in Ref.~\cite{procesi2020note} and Ref.~\cite{Procesi2007LieGroups}.

\begin{theorem}
The invariants of $n \times n$ matrices are generated by elements $\tr(M)$ where $M$ are monomials 
(of degree~$\leq n^2$ by Razmyslov). 
\end{theorem}
Among these invariants the ones that are multilinear and alternating have a very special structure.\smallskip

In fact these invariants have an {\em exterior} multiplication. The algebra of these invariants, 
under exterior  multiplication,  can be identified  as the algebra of multilinear alternating 
functions $(\bigwedge M_d(F)^*)^G$. 
In turn this algebra can be identified to the cohomology of the unitary group. 
As all such cohomology algebras it is a Hopf algebra and  by Hopf's Theorem  it  is the exterior 
algebra generated by the  primitive elements. 
The  primitive elements  of   $(\bigwedge M_d(F)^* )^G$ are~\cite{Kostant2009}:

\begin{equation}\label{prime}
T_{2i-1}=T_{2i-1}(x_1,\ldots,x_{2i-1}):\stackrel{\eqref{st}}=\tr (St_{2i-1}(x_1,\ldots,x_{2i-1}))\,.
\end{equation} 
In particular, since these elements generate an exterior algebra, 
a product of elements $T_i$ is non zero if and only if the $T_i$ involved are all distinct. Given $k$ distinct $T_i$ their product depends on the order only 
up to  sign.

The $2^n$ different products form a basis of  $(\bigwedge M_d(F)^* )^G$.   
In dimension $d^2$ the only non-zero product of these elements containing $d^2$ variables is
 \begin{equation}\label{ILTD}
\mathcal T_d(x_1,x_2,\ldots,x_{d^2})=T_1\wedge T_3\wedge T_5\wedge \cdots \wedge T_{2d-1}.
\end{equation} Notice  that $\tr(St_{2i }(x_1,\ldots,x_{2i}))=0,\ \forall i$.\smallskip

As a consequence, we have:
\begin{proposition}\label{mmu}
Any multilinear  antisymmetric function of  $x_1,\ldots,x_{d^2 }$   is a multiple of \
$T_1\wedge T_3\wedge T_5\wedge \cdots \wedge T_{2d-1}$.
\end{proposition} 

\begin{remark}\label{Iin} 
The function $\det( x_1,\ldots, x_{d^2})$  is an alternating invariant of matrices,
so it must have an expression as in Formula~\eqref{ILTD}.  
In fact the computable integer constant is known up to a sign~\cite{Formanek1987}:
\begin{equation}\label{costdis}\mathcal T_d
(x_1, \dots, x_{d^2})=\mathcal{C}_d\det( x_1,\ldots, x_{d^2}),\quad 
\mathcal C_d:=\pm \frac{1!3!5!\cdots (2d-1)!}{1!2!\cdots (d-1)!}\in \Z.
\end{equation}
\end{remark}

\section{Tensor polynomial identities}\label{sect:TPI}
Let $\lambda = (h_1, \dots, h_n) \vdash k$ denote a partition of $k$ with $n$ parts. 
Consider a permutation $\sigma \in S_n$ that decomposes into $l$ cycles as $\sigma = \mu_1 \cdots \mu_l$. 
We write $j \in \mu$ if position $j$ is moved by cycle $\mu$ and let $\sigma$ act on $\lambda$ in the following manner,
\begin{equation}\label{eq:SnOnLambda}
 \sigma(\lambda) = (\sum_{i \in \mu_1} \lambda_i, \dots, \sum_{i' \in \mu_l} \lambda_{i'} )^{\downarrow}\vdash k
\end{equation}
where $\phantom{}^\downarrow$ means that the elements are arranged in a non-increasing order.
For example, given $\lambda = (6,4,3,1,1)\vdash 15$ and $\sigma = (2,5)(1,3,4)$ one obtains
$$
 \sigma(\lambda) = (4+1, 6+3+1)^{\downarrow} = (10,5)\,.
$$
We say that a partition $\mu$ is a {\em refinement} of $\lambda$ if there is a $\sigma$ such that $\sigma(\mu) = \lambda$.
 
Given two partitions  $\lambda_1 = (h_1, \dots, h_m) \vdash k_1,\  \lambda_2 = (\ell_1, \dots, \ell_n) \vdash k_2$ we set 
$\lambda_1\oplus\lambda_2\vdash k_1+k_2$ the partition $(h_1, \dots, h_m,\ell_1, \dots, \ell_n)^{\downarrow}$.
\bigskip

With these notations we can now state our main Theorem.
\begin{theorem}\label{main}
Let $\lambda\vdash k$. A tensor polynomial $ST(\lambda)(x_1,\ldots,x_k)$ is NOT  a TPI of $d\times d$  matrices if and only if  $k\leq d^2$  and $\lambda\oplus 1^{d^2-k}$  is a refinement of the partition $\delta_d $ of $d^2$ formed by the first $d$ odd integers.
\begin{equation}\label{deld}
\delta_d:=(2d-1,2d-3,\ldots, 9,7,5,3,1).
\end{equation}
\end{theorem}
In particular, this means that the polynomial $ST(\lambda)(x_1,\ldots,x_k),\ \lambda\vdash k$ is NOT  a TPI of $d\times d$  matrices if and only if $ST(\lambda\oplus 1^{d^2-k})(x_1,\ldots,x_d^2) $ is NOT  a TPI.  
The proof will be given later in this Section~\ref{sect:TPI}. 
Thus, contrary to what happens for ordinary polynomial identities, we will need to understand first  the nature of multilinear alternating tensor polynomial $G( x_1,\ldots,x_{d^2})$  in $d^2$ matrix variables.

\bigskip
 
Recall that~\cite{Kostant2009} for the cycle $(k,\dots, 1) = (1, \dots, k)^{-1}$ we have:
\begin{equation}\label{eq:perm_to_traced_prod}
 \tr \big((k,\dots ,1) x_1 \ot \dots \ot x_k \big) = \tr(x_1 x_2 \dots x_k)\,. 
\end{equation}
 and furthermore
$
 T_i = \tr\big( St_i(X)\big)=tr(\sum_{\sigma\in S_i}\epsilon_\sigma \sigma \,  x_1 \ot \dots \ot x_i)
$.

\begin{lemma}\label{lem:tr_sG}
 Let $\lambda = (h_1, \dots, h_n) \vdash k$. 
 For $\sigma \in S_n$ with $l$ cycles, let $\sigma(\lambda) = (\mu_1, \dots, \mu_l)\vdash k$ be the partition of $k$ with $l$ parts as defined in Eq.~\eqref{eq:SnOnLambda}.
 Then for $G:=ST(\lambda)(x_1,\ldots,x_k)$ as in Formula  \eqref{stlam} we have $ \tr( \sigma G )=0$ if one of the $\mu_i$ is even or if the $\mu_i$ are not all distinct, otherwise:
 $$
    \tr( \sigma G ) =    \pm T_{\mu_1} \wedge \dots \wedge T_{\mu_l}\,.
 $$
\end{lemma}
\begin{proof}  In Formula~\eqref{stlam} set 
$$X_1:=x_{1}x_{2}\cdots x_{h_1},\ X_2:= x_{h_1+1}x_{h_1+2}\cdots x_{h_1+h_2},\  \ldots ,X_n:=  x_{k-h_n+1}x_{ k-h_n+2}\cdots x_{ k}  $$
 For each summand in $G$, the trace factorizes over the parts of 
 $\sigma(\lambda)$. 
 
 $$ 
 \tr\big(\sigma ST(\lambda \big) = 
    \tr\big(  \sigma  Alt_X X_1\otimes X_2\otimes\ldots\otimes X_n)= 
  Alt_X  \tr\big(  \sigma  X_1\otimes X_2\otimes\ldots\otimes X_n)
$$
 With the help of Eq.~\eqref{eq:perm_to_traced_prod}, one has
 \begin{equation}\label{mum}
 \tr\big(  \sigma  X_1\otimes X_2\otimes\ldots\otimes X_n)= \tr\big(  M_1)\tr\big(  M_2)\ldots \tr\big(  M_l)
\end{equation} where the $M_i$  are monomials in the $X_j$ and hence monomials in the variables $x_1,\ldots,x_k$  of lengths given by the partition $\sigma(\lambda) = (\mu_1, \dots, \mu_l)\vdash k$. When we alternate we have the exterior product of the alternation of the elements  $ \tr\big(  M_i)$. This is either 0 or,  up to reordering the variables (which may introduce a sign)  equals $T_{\mu_1} \wedge \dots \wedge T_{\mu_l}$  with $\mu_i$ the degree of $M_i$  in the variables $x_1,\ldots,x_k$.
This ends the proof.
 \end{proof}

\begin{corollary}
 If $\max\{\sigma(\lambda)\} \geq 2d$ then $\tr(\sigma ST(\lambda)) = 0$.
\end{corollary}
\begin{proof}
 Then $\tr(\sigma ST(\lambda))$ is proportional to the trace of a standard polynomial in more than $2d$ variables, which by the Amitsur-Levitzky Theorem is known to vanish.
 This ends the proof.
\end{proof}

\begin{proposition}\label{prop:refinement} If $\lambda\vdash d^2$  then 
  $ST(\lambda) \neq 0$ if and only if $\lambda$ is a refinement of $\delta_d = (2d-1, \dots, 3,1)$.
\end{proposition}
\begin{proof}
  $ST(\lambda) = 0$ if and only if $\tr(\sigma ST(\lambda)) = 0$ for all $\sigma \in S_n$. By Lemma~\ref{lem:tr_sG}, 
  $\tr(\sigma ST(\lambda))$ is the alternating product
  $$
       \tr(\sigma ST(\lambda)) = \pm T_{\mu_1} \wedge \dots \wedge T_{\mu_l}\,.
  $$
  By Remark \ref{alv}, this expression vanishes except when $\sigma(\lambda) = \delta_d$,
  i.e. when $\lambda$ is a refinement of~$\delta_d$.
  This ends the proof.
\end{proof}

\medskip

In fact we have a more general statement for any multilinear antisymmetric polynomial in $k \leq d^2$ variables  $f(x_1,\ldots,x_k)$:
\begin{theorem}\label{finalt}
$f(x_1,\ldots,x_k)$ is non zero when evaluated in $d\times d$ matrices if and only if the alternating polynomial in $d^2$ variables  
$$F(x_1,\ldots,x_k,\ldots,x_{d^2}):=Alt_X f(x_1,\ldots,x_k)\otimes x_{k+1}\otimes x_{k+2}\otimes \ldots \otimes x_{d^2}$$ is non zero.
\end{theorem}
\begin{proof} Clearly we only need to show that, if  $f(x_1,\ldots,x_k)$ is non zero, 
then $F(x_1,\ldots,x_k,\ldots,x_{d^2})$ is non zero.  
If  $f(x_1,\ldots,x_k)$ is non zero, then there are $k$ elements $u_1,\ldots,u_k$ out of the canonical basis  $e_{i,j}$ 
so that $f(u_1,\ldots,u_k)\neq 0$. 
Then when we compute  $F(u_1,\ldots,u_k,\ldots,u_{d^2})$ we see that
$$F(u_1,\ldots,u_k,\ldots,u_{d^2})=f(u_1,\ldots,u_k)\otimes u_{k+1}\otimes u_{k+2}\otimes \ldots \otimes u_{d^2}+  O.T.$$ where the other terms $  O.T.$ do not end with $u_{k+1}\otimes u_{k+2}\otimes \ldots \otimes u_{d^2}$  hence $F(u_1,\ldots,u_k,\ldots,u_{d^2})\neq 0$.

\end{proof} This is clearly the content of the main Theorem \ref{main} which we repeat:
\begin{theorem}\label{main1}
Let $\lambda\vdash k$. A tensor polynomial $ST(\lambda)(x_1,\ldots,x_k)$ is NOT a TPI of $d\times d$  matrices if and only if  $k\leq d^2$  and $\lambda\oplus 1^{d^2-k}$  is a refinement of the partition of $d^2$ formed by the first $d$ odd integers.
\begin{equation}\label{deld1}
\delta_d:=(2d-1,2d-3,\ldots, 9,7,5,3,1)\,.
\end{equation}
\end{theorem}

\begin{example}
Figure~\ref{fig:Hasse2} shows all partitions $\lambda \vdash k\leq 9$ for which $ST(\lambda)(x_1, \dots, x_9)$ is NOT a TPI for $3 \times 3$ matrices.
\end{example}

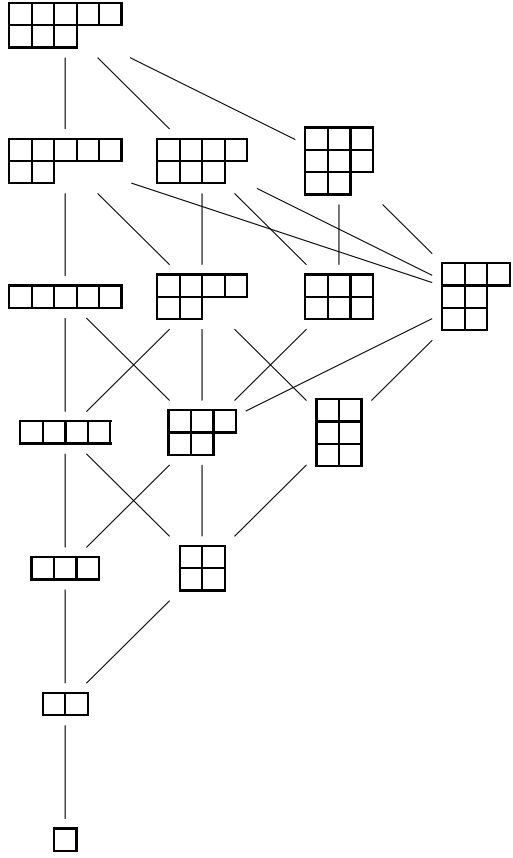
\begin{figure}
\begin{tikzpicture}[scale=1.8]
\ytableausetup{smalltableaux}
 \node (53) at (0,0) {\ydiagram{5,3}};
 \node (52) at (0,-1) {\ydiagram{5,2}};
 \node (5) at (0,-2) {\ydiagram{5}};
 \node (4) at (0,-3) {\ydiagram{4}};
 \node (3) at (0,-4) {\ydiagram{3}};
 \node (2) at (0,-5) {\ydiagram{2}};
 \node (1)  at (0,-6) {\ydiagram{1}};
 
 \node (43) at (1,-1) {\ydiagram{4,3}};
 \node (42) at (1,-2) {\ydiagram{4,2}};
 \node (32) at (1,-3) {\ydiagram{3,2}};
 \node (222) at (2,-3) {\ydiagram{2,2,2}};
 
 \node (22) at (1,-4) {\ydiagram{2,2}};
 
 \node (332) at (2,-1) {\ydiagram{3,3,2}};
 
 \node (33) at (2,-2) {\ydiagram{3,3}};
 \node (322) at (3,-2) {\ydiagram{3,2,2}};
 
 \draw (53) -- (52) -- (5) -- (4) -- (3) -- (2) -- (1);
 \draw (53) -- (43) -- (42) -- (32) -- (22);
 
 \draw (42) -- (4);
 \draw (22) -- (2);

 \draw (42) -- (222) -- (22);
 
 \draw (53) -- (332) -- (33);
 \draw (33) -- (32) -- (3);
 \draw (332) -- (322) -- (222);
 \draw (322) -- (32);
 \draw (43) -- (33);
 \draw (52) -- (42);
 \draw (4) -- (22);
 \draw (5) -- (32);
 \draw (43) -- (322);
 \draw (52) -- (322);
 \end{tikzpicture}
\caption{
\label{fig:Hasse2}
Partitions such that $\lambda \oplus 1^{d^2-k}$ are refinements of $(5,3,1)$,
corresponding to the all alternating tensor polynomials $ST(\lambda)$ that are NOT TPI in $k < d^2 = 9$ variables.
Two Young tableaux are connected if one of them is a refinement of the other.
}
\end{figure} 

Let us consider an interesting class of TPI formed by rectangular partitions, namely $ST(m^n)$ as defined in Eq.~\eqref{ba}.
Let $d,m\leq 2d$ be two integers and define
$$ [d;m]=[(2d-1)/m]+ [(2d-3)/m]+\ldots+[9/m]+[7/m]+[5/m]+[3/m]+[1/m]$$where $[x]$  denotes as usual the {\em integer part} of $x$.
For example
$$[2;2]=1,\  [3;2]=3,\ [5;2]=10,\ [2;3]=1;\  [3;3]=2,\ [5;3]=7.$$
\begin{proposition}\label{ad}
$ST(m^n)$ is a TPI  for $d\times d$ matrices if and only if $n> [d;m]$.
\end{proposition}
 \begin{proof}
Assume that $m,n$ are such that $ST(m^n)$  is not a TPI for $d\times d$  matrices, then, by  Theorem \ref{main1}  we have $mn\leq d^2$  and the partition $m^n\oplus 1^{d^2-mn}$ is a refinement of $\delta_d$.  This means that  for each $i=1,\dots,d$  we have  $2i-1=mr_i+(2i-1-mr_i),\ mr_i\leq 2i-1$ and $\sum_ir_i=n$. In other words $r_i\leq [(2i-1)/m]$, and hence $n=\sum_ir_i\leq  [d; m]  $.

The converse is then clear.
\end{proof}

\begin{table}
 \begin{tabular}{@{}lll@{}}
\toprule
$d$ & $\lambda$\\
\midrule
$2$ &
$ 2 ^ 2 $,
$ 3 ^ 2 $,
$ 4 ^ 1 $
\\
$3$ &
$ 2 ^ 4 $,
$ 3 ^ 3 $,
$ 4 ^ 2 $,
$ 5 ^ 2 $,
$ 6 ^ 1 $
\\
$4$ &
$ 2 ^ 7 $,
$ 3 ^ 5 $,
$ 4 ^ 3 $,
$ 5 ^ 3 $,
$ 6 ^ 2 $,
$ 7 ^ 2 $,
$ 8 ^ 1 $
\\
$5$ & 
$ 2 ^ {11} $,
$ 3 ^ 8 $,
$ 4 ^ 5 $,
$ 5 ^ 4 $,
$ 6 ^ 3 $,
$ 7 ^ 3 $,
$ 8 ^ 2 $,
$ 9 ^ 2 $,
$ 10 ^ 1 $
\\
$6$ & 
$ 2 ^ {16}$,
$ 3 ^ {11} $,
$ 4 ^ 7 $,
$ 5 ^ 6 $,
$ 6 ^ 4 $,
$ 7 ^ 4 $,
$ 8 ^ 3 $,
$ 9 ^ 3 $,
$ 10 ^ 2 $,
$ 11 ^ 2 $,
$ 12 ^ 1 $,
\\
$7$ & 
$ 2 ^ {22} $,
$ 3 ^ {15} $,
$ 4 ^ {10} $,
$ 5 ^ 8 $,
$ 6 ^ 6 $,
$ 7 ^ 5 $,
$ 8 ^ 4 $,
$ 9 ^ 4 $,
$ 10 ^ 3 $,
$ 11 ^ 3 $,
$ 12 ^ 2 $,
$ 13 ^ 2 $,
$ 14 ^ 1 $\\
$8$ &
$ 2 ^ {29} $,
$ 3 ^ {20} $,
$ 4 ^ {13} $,
$ 5 ^ {11} $,
$ 6 ^ 8 $,
$ 7 ^ 7 $,
$ 8 ^ 5 $,
$ 9 ^ 5 $,
$ 10 ^ 4 $,
$ 11 ^ 4 $,
$ 12 ^ 3 $,
$ 13 ^ 3 $,
$ 14 ^ 2 $,
$ 15 ^ 2 $,
$ 16 ^ 1 $
\\
$9$ &
$ 2 ^ {37} $,
$ 3 ^ {25} $,
$ 4 ^ {17} $,
$ 5 ^ {14} $,
$ 6 ^ {10} $,
$ 7 ^ 9 $,
$ 8 ^ 7 $,
$ 9 ^ 6 $,
$ 10 ^ 5 $,
$ 11 ^ 5 $,
$ 12 ^ 4 $,
$ 13 ^ 4 $,
$ 14 ^ 3 $,
$ 15 ^ 3 $,
$ 16 ^ 2 $,
$ 17 ^ 2 $,
$ 18 ^ 1 $\\
\bottomrule
\end{tabular}
\medskip
\caption{\label{tab:mn_TPIs}
Partitions $m^n$ for which $ST(m^n)$ is a TPI in small dimensions. 
Here partitions made up of equal-sized parts are written as
$m^n = (m, \dots, m)$ \, ($n$ times).
Note that if $\lambda = m^n$ is a TPI, then so are $(m+j)^n$ and $m^{(n+j)}$ for $j\geq 1$ and thus they are not listed.
}
\end{table}

\section{Minimal Identities} 
Given a TPI $g(x_1,\ldots,x_k)$  one can {\em deduce} from $g$  other TPI's  by the following procedures:

1) Substitute to the variables $x_i$ some polynomial in $F\langle X\rangle$ (in particular permuting the variables).  

2)  Multiply to right and left by tensor polynomials. 

3) Permute the tensor factors.  

4) Take linear combination  of previously found TPI's.

\smallskip

The problem of understanding minimal TPI's is quite open. 
We approach this problem by focusing on special cases.

\begin{definition}\label{mint}
A TPI is {\em minimal} if it cannot be deduced from TPI's of strictly lower degree.

We will say that a TPI $ST(\lambda)$ is {\em $\lambda$-minimal} if for any partition $\lambda'$ obtained from $\lambda$ by removing one box, $ST(\lambda')$ is not a TPI.
\end{definition}If on the other hand  $ST(\lambda')$ is  a TPI then $ST(\lambda)$ can be deduced from $ST(\lambda')$ as follows.
Assume that the box removed is on the $i^{th}$ row and $\lambda\vdash k,\ ht(\lambda)=n$, then multiply  $ST(\lambda')$ (which is now in the variables $x_1,\ldots,x_{k-1}$) by the {\em tensor variable } $1^{\otimes i-1}\otimes x_k\otimes 1^{\otimes n-i }$ and then alternate the result as
$$ST(\lambda)=\frac 1{(k-1)!}Alt_{x_1,\ldots,x_ k  }ST(\lambda')1^{\otimes i-1}\otimes x_k\otimes 1^{\otimes n-i }. $$

For two partitions $\lambda\vdash h,\ \mu\vdash k$ write $\lambda \subset \mu$ if $\lambda_i \leq \mu_i$ for all parts $i$, i.e. the Young diagram of $\lambda$ fits into that of $\mu$. The above discussion establishes that:

\begin{proposition}\label{lem:TPIinclusion}\ \\
 1. If $ST(\lambda)$ is a TPI and $\lambda \subset \mu$ then $ST(\mu)$ is a TPI. \\
 2. If $ST(\lambda)$ is {\em not} a TPI and $\mu\subset \lambda$ then $ST(\mu)$ is {\em not} a TPI.
\end{proposition}

By Proposition~\ref{ad}, given two integers $d,m\leq 2d$, the tensor polynomial
 $ST(m^{[d;m]+1})$   is a TPI  for $d\times d$ matrices (cf. Table~\ref{tab:mn_TPIs}).
One may ask whether a given TPI $ST(m^{[d;m]+1})$ is also $\lambda$-minimal, that is, 
whether $ST(m^{[d;m]} \oplus m-1)$ is not also a TPI for $d\times d$ matrices.
 \begin{theorem}\label{bett}
 $ST(m^{[d;m]} \oplus (m-1))$ is {\em not} a TPI  for $d\times d$ matrices if  and only if $m$ is even or if $m$ is odd and   $m\leq d $.
\end{theorem}
\begin{proof}

 From Proposition~\ref{ad} we know that  $ST(m^n)$ is a TPI if and only if $n>[d;m]$.
 Therefore $ST(m^{[d;m]}) $  is {\em not} a TPI and $m^{[d;m]} \oplus 1^{d^2 - [d;m]}$ is a refinement of $\delta_d$.
  
 Consider first the case of $m=2l$  even for some positive integer $l\geq 1$.
 We have that $[2i-1/2l]=0$  for all  $i\leq l$ and for $i=l$ one has $2l-1=m-1$.
 Looking at the proof of  Proposition~\ref{ad}  we see that the refinement $m^{[d;m]}$ is obtained by writing all terms 
 $2i-1= m  r_i+1\cdot (2i-1-mr_i)  $ as a sum of  $r_i=[2i-1/m]$  terms equal to $m$ and $ 2i-1-mr_i   $  terms equal to $1$. 
 This decomposition involves the terms $m$  only  if $r_i>0$, that is when $2i-1>m-1\iff i>l$. Then  for the partition  $m^{[d;m]} \oplus (m-1) \oplus 1^{d^2-m {[d;m]}-(m-1)}$  we can still write  $2i-1= m  r_i+(2i-1-mr_i)$ for all $i>l$  while we write $2l-1=m-1$. The other terms are just sums of 1 and we have that this partition is also a refinement of  $\delta_d$.
 
 \smallskip
 
 Assume now $m=2l-1$  odd with $l\geq 2$. We have that $[2i-1/m]=0$  for all  $i< l$ and for $i=l$  we have $2l-1=m $. 
 Again we know that  $ST(m^{[d;m]})$  is not a TPI  since $m^{[d;m]}\oplus 1^{d^2-m  [d;m] }$ is a refinement of  $\delta_d$. We can and also have to use all the $[d;m]$ copies of $m$  to refine the part  of $\delta_d$ formed by the elements $2i-1\geq m$. This forces to write  $2i-1= m  r_i+(2i-1-r_i),\ i\geq l  $ since by definition $[d;m]=\sum_i r_i$. 
 
 Considering $m^{[d;m]} \oplus (m-1)$, 
 the part of size $m-1$ can be added to the refinement if and only if $m-1\leq (2i-1-mr_i)$ for some $i\geq l$. 
 Now $(2i-1-mr_i)$ is the remainder of the euclidean division of  $2i-1$  by $m$.  
 In particular, the remainder is $=m-1$ in the case of $i=m$ when $d\geq m$. Then choose $r_m=1$ and one has $(2i-1-mr_m)=m-1$ 
 and we may write  $2m-1$  as $m+m-1$.

 However, in the case of $m$ odd with $m>d$ then $m-1$ cannot be added to the largest $2i-1$ which is $2d-1$. 
 Then $m^{[d;m]}\oplus 1^{d^2-m  [d;m] }$ is not a refinement of  $\delta_d$ and $ST(m^{[d;m]}\oplus 1^{d^2-m  [d;m] })$ is a TPI.
 This ends the proof.
\end{proof}
\begin{lemma}\label{mmin}
If $n\leq [\frac d2]+1$ then $[d;2(d-n+1)]=n-1$.
\end{lemma}
\begin{proof}
Let us compute $[\frac{2i-1}{2(d-n+1)}]$.  If $2(d-n+1)>2i-1$ that is $i\leq d-n+1$  then $[\frac{2i-1}{2(d-n+1)}]=0$. 

We have  $4(d-n+1)>2d-1$ is equivalent to   $2(d-2n)+5>0\iff  d-2n \geq -2\iff n-1\leq [d/2]$ therefore  for $d\geq i> d-n+1$ we have $[\frac{2i-1}{2(d-n+1)}]=1.$ Since there are $n-1$ terms if this type the claim  follows.
\end{proof}
Define $\lambda_{d,n} := (2(d-n +1))^n$.
\begin{corollary}
 A TPI $ST(\lambda)$ for $d\times d$ matrices in $n$ tensor factors with $n \leq [d/2]+1$ is $\lambda$-minimal
 if and only if 
 $\lambda = \lambda_{d,n}$ (then it is in in $k=n[2(d-n +1)]$ variables).
\end{corollary}
\begin{proof}
  The fact that  $ST(\lambda_{d,n} )$  is a $  \lambda$-minimal  TPI 
 follows from Theorem \ref{bett} and the previous Lemma.

Now it remains to show the {\em only if}: why is a  $\lambda$-minimal TPI in $n \leq [d/2]+1$ non trivial tensor factors that particular rectangular partition?

  If $ST(\lambda)$ with $\lambda\neq  \lambda_{d,n} $ is a $\lambda$-minimal TPI then $\lambda$   cannot contain the rectangle $ (2(d-i +1))^i,\ \forall  i\leq n$, since  $ST(\lambda_{d,n} )$  is $\lambda$-minimal and $ST(\lambda_{d,i<n})$ are not TPI's. 
  So we must have $\lambda_i\leq 2(d-i +1)-1=2d-(2i-1),\ 1\leq i\leq n$. 

Thus $\lambda$  is contained in 
$$2 d-1, 2 d-3,\ldots, 2d-(2n-1)  $$ but this partition  can be completed to  $\delta_d$ so it is not a partition of a TPI.

\end{proof}

\section{Tensor polynomials and the Weingarten function}
\label{sect:inWein}

We now explain how to evaluate arbitrary alternating tensor polynomials in $d^2$ variables. 
Then $ST(\lambda)$ with $\lambda \vdash n$ is in $\Sigma_n(F^d)$, the algebra spanned by the permutations in $\End((F^d)^{\otimes n})$.
However by tolerating a slight abuse of notation we can regard $ST(\lambda)$ also as an element of the group algebra~$\C[S_d]$.
 
For later use, we recall that the central units  of  $ \C[S_d] $ (in fact already in  $ \Q[S_d] $) are the elements
$$
 \omega_\mu = \frac{\chi_\mu(e)}{d!} \sum_{\sigma \in S_d} \chi_\mu(\sigma^{-1}) \sigma\,, \quad \mu \vdash d
$$
form a basis for the center $Z(\C[S_d])$.

We also recall the Schur-Weyl decomposition 
$$
 (\C^{d})^{\otimes n} \cong \bigoplus_{\mu \vdash n} U_\mu \ot S_\mu\,.
$$
with $U_\mu$ (resp. $S_\mu$) irreducible representations of $S_d$ (resp. $GL(d,\C)$).
Then 
$$
 \tr(\omega_\mu) = \dim(U_\mu) \dim(S_\mu) = \chi_\mu(e) s_{\mu,d}(1) \,,
$$
where $s_{\mu, d}(x)$ is the Schur polynomial $s_\mu(x, \dots, x)$ in $d$ variables.

\bigskip
Let us now treat the case ${\delta_d} = (2d-1, \dots, 3, 1)$.

\begin{proposition}
 Let $\delta_d= (2d-1,2d-3, \dots, 3, 1)$ and $G_d=ST(\delta_d)(x_1, \dots, x_{d^2})$ with $x_1, \dots, x_{d^2} \in M_d(F)$. Then 
\begin{align}\label{eq:special_case}
  G_d &= \det(x_1, \dots, x_{d^2})J_d, \\
  J_d &=\mathcal{C}_d   \frac{1}{d!} \sum_{\substack{\mu \vdash d\\l(\mu) \leq d}} 
                                    \frac{\chi_\mu(e) }{s_{\mu,d}(1)} \,\omega_\mu \, ,\quad \mathcal C_d:=\pm \frac{1!3!5!\cdots (2d-1)!}{1!2!\cdots (d-1)!}.\nn
 \end{align} 
\end{proposition}

\begin{proof}
The operator $G_d$ is completely determined by the set of inner products $\{ \tr(\sigma G_d) \,|\, \sigma \in S_d\}$. 
Proposition~\ref{prop:refinement} implies that
\begin{equation}\label{eq:GexpectSigma}
 \tr(\sigma G_d) = \begin{cases}
                         \mathcal{C}_d \det(x_1, \dots, x_{d^2}) & \text{if } \sigma = e\\
                         0    & \text{otherwise}
                         \end{cases}\iff  \tr(\sigma J_d) = \begin{cases}
                         \mathcal{C}_d   & \text{if } \sigma = e\\
                         0    & \text{otherwise}
                         \end{cases}
\end{equation}

Now observe that $ \tilde{G}_d := (d!)^{-1} \sum_{\pi \in S_d} \pi G_d \pi^{-1}$
has the same set of inner products with permutations $\sigma$ as $G_d$ does. 
Therefore $\tilde{G}_d = G_d$ and the operator  $J_d$ lies in the center $Z(\C[S_d])$.
Equation~\eqref{eq:GexpectSigma} yields
\begin{align}
  \tr(\omega_\mu J_d)  &= \tr \big(\frac{\chi_\mu(e)}{d!} 
                        \sum_{\sigma \in S_d} \chi_\mu(\sigma^{-1}) \sigma J_d) \big)
                     = \frac{\chi_\mu(e)^2}{d!} \mathcal{C}_d  \,.\nn
\end{align}
Then
\begin{align}
    J_d &= \sum_{\substack{\mu \vdash d \\ l(\mu)\leq d}} \frac{\tr(\omega_\mu J_d)}{\tr(\omega_\mu)}   \,\omega_\mu 
              = \mathcal{C}_d \frac{1}{d!} \sum_{\substack{\mu \vdash d \\ l(\mu)\leq d}}  
                    \frac{\chi_\mu(e) }{s_{\mu,d}(1)} \,\omega_\mu\,.\nn
\end{align} 
This ends the proof.
\end{proof}

\begin{table}[tbp]
\caption{\label{fig:tables}
Expansion of $J_d:=J_G$ for $G_d=ST((2d-1, 2d-3, \dots, 3,1))$
for $d\leq 6$.
Here $\omega_\mu$ are the minimal central idempotents and $c_\mu = \sum_{\text{cycle type}(\pi) = \mu} \pi $ is the sum over all permutations of cycle type $\mu$.
The overall sign of $J_d $ depends on that of $\mathcal{C}_d$ which is presently unknown.
In the second table we write the cycle types in non-increasing order. Then we note that the absolute values $|c_\mu|$
are increasing in lexicographic order.}
\begin{tabular}{@{}ll@{}}
$d$ \quad\quad & $J_d$ expanded in $\omega_{\mu}$ \\
\toprule
2 & $  
  \omega_{ [2] } +
3 \omega_{ [1, 1] } 
$ \\
3 &$ 3 \big(
2 \omega_{ [3] } +
5 \omega_{ [2, 1] } +
20 \omega_{ [1, 1, 1] } 
\big)$ \\
4 &$ 60 \big(
6 \omega_{ [4] } +
14 \omega_{ [3, 1] } +
21 \omega_{ [2, 2] } +
42 \omega_{ [2, 1, 1] } +
210 \omega_{ [1, 1, 1, 1] } 
\big)$ \\
5 &$ 15120 \big(
20 \omega_{ [5] } +
45 \omega_{ [4, 1] } +
72 \omega_{ [3, 2] } +
120 \omega_{ [3, 1, 1] } +
168 \omega_{ [2, 2, 1] } +
420 \omega_{ [2, 1, 1, 1] } $ \\& $+
2520 \omega_{ [1, 1, 1, 1, 1] } 
\big)$ \\
6 & $ 10886400 \big(
420 \omega_{ [6] } +
924 \omega_{ [5, 1] } +
1540 \omega_{ [4, 2] } +
2310 \omega_{ [4, 1, 1] } +
1980 \omega_{ [3, 3] } $ \\& $+
3465 \omega_{ [3, 2, 1] }  +
6930 \omega_{ [3, 1, 1, 1] } +
5544 \omega_{ [2, 2, 2] } +
9240 \omega_{ [2, 2, 1, 1] } $ \\& $+
27720 \omega_{ [2, 1, 1, 1, 1] }  +
194040 \omega_{ [1, 1, 1, 1, 1, 1] } 
\big)$
\\
\bottomrule \\
$d$ \quad\quad & $J_d$ expanded in $c_\mu$ \\
\toprule
2 & $
- c_{ [2] }
+ 2 c_{ [1, 1] }
$
\\
3 & $ 3 \big(
  2 c_{ [3] }
-3 c_{ [1, 2] }
+ 7 c_{ [1, 1, 1] }
\big)$
\\
4 & $ 15 \big(
-20 c_{ [4] }
+ 22 c_{ [2, 2] }
+ 29 c_{ [1, 3] }
-48 c_{ [1, 1, 2] }
+ 134 c_{ [1, 1, 1, 1] }
\big)$
\\
5 & $ 2520 \big(
  70 c_{ [5] }
-74 c_{ [2, 3] }
-101 c_{ [1, 4] }
+ 115 c_{ [1, 2, 2] }
+ 160 c_{ [1, 1, 3] }
-299 c_{ [1, 1, 1, 2] } $ \\& $
+ 1015 c_{ [1, 1, 1, 1, 1] }
\big)$
\\
6 & $ 1814400 \big(
-882 c_{ [6] }
+ 900 c_{ [3, 3] } 
+ 922 c_{ [2, 4] }
-1014 c_{ [2, 2, 2] }
+ 1274 c_{ [1, 5] } $ \\ & $
-1377 c_{ [1, 2, 3] } 
-2004 c_{ [1, 1, 4] } 
+ 2396 c_{ [1, 1, 2, 2] } 
+ 3540 c_{ [1, 1, 1, 3] } 
-7614 c_{ [1, 1, 1, 1, 2] } $ \\ & $
+ 31524 c_{ [1, 1, 1, 1, 1, 1] }
\big)$
\\
\bottomrule
\end{tabular}
\end{table}

Up to the sign of $\mathcal{C}_d$, the values of $G_d = ST(\delta_d)$ for small $d$ are listed in Table~\ref{fig:tables}. 
Here, $\omega_\mu$ are the minimal central idempotents and 
$c_\mu = \sum_{\text{cycle type}(\pi) = \mu} \pi $ is the sum over all permutations of cycle type $\mu$.

We remark on an interesting property: consider the cycle types written in a non-decreasing fashing. 
When expanded in terms of $c_\mu$ the absolute values $|c_\mu|$ 
are increasing in lexicographic order (cf. \cite{procesi2020note}).

\bigskip We now introduce the Weingarten function.
Let $A \in End((\C^d)^{\ot n})$ and define the maps
\begin{align}
 \mathds{E}(A)  &= \int_{\mathcal{U}(d)} U^{\otimes n} A \,(U^{-1})^{\otimes n} dU\,, dU\quad\text{Haar measure} \nn\\ 
 \Phi(A)        &= \sum_{\sigma \in S_n} \tr(\sigma^{-1} A) \sigma\,. \nn
\end{align}
$\mathds{E}(A)$ is also known as a $n$-fold twirl or a twirling channel.
It is known that~\cite{CollinsSniady2006}
\begin{align}\label{mma}
 \mathds{E}(\one) &= \one \,,\\
 \Phi(\one)       &= n! \sum_{\substack{\mu \vdash n \\ l(\mu) \leq d}} \frac{s_{\mu, d}(1)}{\chi_\mu(e)} \omega_\mu\,,\nn
\end{align} 
The relation between $\mathds{E}$ and $\Phi$ is the following~\cite[Prop. 2.3]{CollinsSniady2006}:
\begin{equation}\label{eq:Phi_to_E}
 \Phi(A) = \mathds{E}(A) \Phi(\one)\,.
\end{equation}
Since the elements  $ \omega_\mu$ are central units, adding to $\one$ in  the operator algebra $  End((\C^d)^{\ot n})$, the operator  $\Phi(\one)$ of  Formula \eqref{mma}, is invertible. The  inverse of $\Phi(\one)$  also known as the Weingarten operator, is
\begin{align}\label{eq:weinOp}
 \Wg(d,n) = \frac{1}{(n!)^2} \sum_{\substack{\mu \vdash n \\ l(\mu) \leq d}} \frac{\chi_\mu(e)^2}{s_{\mu,d}(1)} \sum_{\pi \in S_k} \chi_\mu(\pi^{-1})\pi = \frac{1}{n!} \sum_{\substack{\mu \vdash n \\ l(\mu) \leq d}} \frac{\chi_\mu(e)}{s_{\mu,d}(1)} \omega_\mu\,.
\end{align} 
Observe that if $d\geq n$ then   $\Phi(\one)       = n! \sum_{  \mu \vdash n  } \frac{s_{\mu, d}(1)}{\chi_\mu(e)} \omega_\mu $  is invertible as element of the algebra of the symmetric group, which in this case it is isomorphic to the algebra of operators which it induces on $  End((\C^d)^{\ot n})$.
And thus solving for $\mathds{E}(A)$ in Eq.~\eqref{eq:Phi_to_E} gives
$$
 \mathds{E}(A) = \Phi(A) \Wg(d,n)\,.
$$

Comparing Eq.~\eqref{eq:special_case} and Eq.~\eqref{eq:weinOp} we see that:
\begin{corollary}
Let $\delta_d = (1,3, \dots, 2d-1)$ and $x_1, \dots, x_{d^2} \in M_d(F)$. Then 
 \begin{equation}
  ST(\delta_d)(x_1, \dots, x_{d^2}) = \mathcal{T}_d (x_1, \dots, x_{d^2})\Wg(d,d)= \mathcal{C}_d \det(x_1, \dots, x_{d^2})\Wg(d,d)\,. \nn
 \end{equation} 
\end{corollary}

\bigskip

We can now expand general $ST(\lambda)$ in $d^2$ variables with $\Wg$.
If $A$ is a linear combination of permutations then $\mathds{E}(A) = A$ and we obtain the expansion

$$
 A = \sum_{\sigma \in S_n} \tr(\sigma^{-1} A) \sigma \Wg(d,n)\,.
$$
In particular, for $ST(\lambda) := ST(\lambda)(x_1, \dots, x_{d^2})$ we have
\begin{align}
 ST(\lambda) &=\sum_{\sigma \in S_n} \tr\big(\sigma^{-1} ST(\lambda)\big) \sigma \Wg(d,n) \nn \\
 &= \det(x_1,\ldots,x_{d^2})J_\lambda ,\quad J_\lambda \in \Sigma_n(F^d)\subset M_d(F)^{\otimes n} \,. \nn
\end{align}

Evaluating this expression (that is finding $J_\lambda$) is a matter of knowing the values $\tr(\sigma^{-1} ST(\lambda))$.
By Proposition~\ref{mmu} and Lemma~\ref{lem:tr_sG}, these are given by
\begin{align}\label{fof}
    \tr( \sigma^{-1} ST(\lambda) ) &=    \pm T_{\mu_1} \wedge \dots \wedge T_{\mu_l} \\
    &=
    \begin{cases}
     \pm \mathcal \mathcal{C}_d\det( x_1,\ldots, x_{d^2})   &\text{ if all $\mu_i$ are odd and distinct } \nonumber\\
     0 &\text{else}.
    \end{cases}
\end{align}
with $(\mu_1, \dots, \mu_l) = \sigma^{-1}(\lambda)$.

The sign of each term $\tr(\sigma^{-1} ST(\lambda))$ is determined by the order of the variables appearing in the monomials, 
that is by the sign of the permutation needed to obtain
$$
    \mathcal T_d = T_1\wedge T_3\wedge T_5\wedge \cdots \wedge T_{2d-1} = \mathcal{C}_d\det( x_1,\ldots, x_{d^2})\,,
$$
$$
T_1:= \tr(x_1) \,,
T_3:= \tr(x_2 x_3 x_4) \,, 
\dots, 
T_{2d-1}:= \tr(x_{d^2-2d+2} x_{d^2-2d+3} \cdots x_{d^2})\,.
$$
so that length of the monomials under the traces are $1, 3, 5, \dots, 2d-1$.

\bigskip

{\scriptsize
\begin{table}[tbp]
\caption{\label{fig:TP}
$\Phi(J_\lambda / \mathcal{C}_d) = \Phi(ST(\lambda)/\mathcal T_d)$ and $J_\lambda = ST(\lambda) / \det( x_1,\ldots, x_{d^2})$ for some $\lambda$.}
\begin{tabular}{@{}ll@{\quad}l@{}}
$d$ & $\lambda$ \quad\quad & $  \Phi(J_\lambda/\mathcal C_d) $\\
\toprule
$2$ &$[3, 1]$    & $()$ \\
    &$[2, 1, 1]$ &  $(1,2) - (1,3)$ \\
    &$[1,1,1,1]$ & $(2,3,4) - (2,4,3) - (1,2,3) + (1,2,4) + (1,3,2) - (1,3,4) - (1,4,2) + (1,4,3)$\\\addlinespace[0.3mm]
$3$ &$[5, 3, 1]$    & $()$\\
    &$[5, 2, 1, 1]$ &  $(2,3) - (2,4)$\\
    &$[4, 3, 1, 1]$ &  $-(1,3) + (1,4)$\\
    &$[3,3,2,1]$    & $-(2,3) + (1,3)$ \\
    &$[3,2,2,1,1]$  & $-(2,3,4) + (2,3,5) - (2,4,3) + (2,5,3) + (1,2)(3,4) - (1,2)(3,5) + (1,3)(2,4) $ \\ && $- (1,3)(2,5)$\\
\bottomrule
   \\
$d$ & $\lambda$ \quad\quad & $\pm J_\lambda$\\
\toprule
$2$ &$[3, 1]$    & $2() - (1,2)$ \\
    &$[2, 1, 1]$ & $(1,2) - (1,3)$ \\
    &$[1,1,1,1]$ & $\frac{1}{2}[(2,3,4) - (2,4,3) - (1,2,3) + (1,2,4) + (1,3,2) - (1,3,4) - (1,4,2) + (1,4,3)]$\\ \addlinespace[0.3mm]
$3$ &$[5, 3, 1]$    & $3[7() - 3(2,3) - 3(1,2) + 2(1,2,3) + 2(1,3,2) - 3(1,3)]$\\
    &$[5, 2, 1, 1]$ & $8(2,3) - 8(2,4) - 3(1,2,3) + (1,2,3,4) - (1,2,4,3) + 3(1,2,4) - 3(1,3,2) $\\
    &&$- (1,3,4,2) + (1,3) + 3(1,3)(2,4) + (1,4,3,2) + 3(1,4,2) - (1,4) - 3(1,4)(2,3)$\\
    &$[4, 3, 1, 1]$ & $-(2,3) + (2,4) + 3(1,2,3) + (1,2,3,4) - (1,2,4,3) - 3(1,2,4) + 3(1,3,2) $\\
    &&$- (1,3,4,2) - 8(1,3) + 3(1,3)(2,4) + (1,4,3,2) - 3(1,4,2) + 8(1,4) - 3(1,4)(2,3)$ \\
    &$[3,3,2,1]$    & $-8(2,3) + 3(2,3,4) + 3(2,4,3) - (2,4) - (1,2,3,4) + (1,2,4,3) + (1,3,4,2) $ \\
    &&$+ 8(1,3) - 3(1,3,4) - 3(1,3)(2,4) - (1,4,3,2) - 3(1,4,3) + (1,4) + 3(1,4)(2,3)$ \\
    &$[3,2,2,1,1]$  & $(3,4) - (3,5) - 3(2,3,4) + 3(2,3,5) - 3(2,4,3) + (2,4) + 3(2,5,3) - (2,5) $\\
    &&$+ 3(1,2)(3,4) - 3(1,2)(3,5) - (1,2,4,3) + (1,2,4)(3,5) + (1,2,5,3) $ \\ 
    &&$- (1,2,5)(3,4) - (1,3,4,2)   + (1,3,5,2) + 3(1,3)(2,4) - (1,3,5)(2,4) $\\
    &&$- 3(1,3)(2,5) + (1,3,4)(2,5) + (1,4,2)(3,5) - (1,4)(2,3,5)  + (1,4,3)(2,5) $\\
    &&$- (1,4)(2,5,3) - (1,5,2)(3,4) + (1,5)(2,3,4) - (1,5,3)(2,4) + (1,5)(2,4,3)$\\
  \bottomrule
\end{tabular}
\end{table}
}

As an example, the partition $\lambda = [5,3,3,2,2,1]\vdash 4^2$  gives ($n=6,\ d=4$)
\begin{equation}
\begin{split}
 \Phi(ST(\lambda)) = (3,4,5) + (3,5,4) - (2,4,5) - (2,5,4) - (1,4)(3,5) \\ + (1,4)(2,5) - (1,5)(3,4) + (1,5)(2,4)\,.  \nn
\end{split}
\end{equation}

In Table~\ref{fig:TP} we list the elements $\Phi(J_\lambda/\mathcal C_d)$ 
and $J_{\lambda}$ for some non-vanishing tensor polynomials  $ST(\lambda)$ in small dimensions. 
Notice that Formula  \eqref{fof} exhibits  $\Phi(J_\lambda/\mathcal C_d)$ as a linear combination of the permutations of $n$ with exactly $d$  cycles and coefficient  $\pm 1,0$.

\FloatBarrier

\medskip

 We may apply Proposition \ref{cse},  and obtain further non trivial TPI.
 For example, one has that
$$
ST((2,1,1)) - ST((1,2,1)+ ST((1,1,2))  
$$
is an identity on $2\times 2$ matrices in exactly $3$ tensor factors. 
Taking a single $ST(\lambda)$ however, no identity exists in $3$ tensor factors.\bigskip

\bibliographystyle{amsplain}
\bibliography{current_bib}
\end{document}